\documentclass[12pt,a4paper]{article}
\usepackage{amsmath}
\usepackage{amsthm}
\usepackage{amssymb}
\usepackage{graphicx}
\usepackage{tabularx}
\usepackage{color}
\usepackage{setspace}
\usepackage{capt-of}
\usepackage{float}
\usepackage{cite}
\usepackage{sectsty}
\usepackage{listings}
\usepackage{fancyhdr}
\usepackage[gen]{eurosym}
\usepackage{multirow}
\usepackage{textcomp}
\usepackage{fancyvrb}
\usepackage{subfig}
\usepackage{paralist}
\usepackage{epic}
\usepackage{mathtools}
\usepackage[mathscr]{euscript}
\usepackage{mathrsfs}
\newtheorem{definition}{Definition}[section]
\newtheorem{Example}[definition]{Example}

\newtheorem{thm}[definition]{Theorem}

\newtheorem{cor}[definition]{Corollary}

\newtheorem{lem}[definition]{Lemma}

\usepackage[top=2.5cm,bottom=2.5cm,left=2cm,right=2cm]{geometry}
\usepackage{multicol}
\usepackage{hyperref}
\usepackage{setspace}
\hypersetup{
	colorlinks,
	citecolor=blue,
	filecolor=black,
	linkcolor=black,
	urlcolor=black
}
\date{}
\begin{document}
	\clearpage
	\thispagestyle{empty}
	
	\setcounter{page}{1}
	\pagenumbering{roman}
	\newpage
	\begin{center}
		\textbf{Bounded Composition Operators on Hilbert Space of Complex-Valued Harmonic Functions}
	\end{center}
\vspace{2mm}
	\begin{center}
		${Tseganesh}$ ${Getachew}$ ${Gebrehana}^{1}$ and ${Hunduma}$ ${Legesse}$ ${Geleta}^{2}$ 	
		\\\textbf{tseganesh.getachew@aau.edu.et and hunduma.legesse@aau.edu.et}\\
		\vspace{8mm}
		\textbf{Department of Mathematics, College of Natural and Computational Sciences, Addis Ababa University, Addis Ababa, Ethiopia.}
	\end{center}	     
	\pagenumbering{arabic} 
	\vspace{6mm}
	\textit{\textbf{Abstract.} In this paper, we study  composition operators on Hilbert space of complex-valued harmonic functions. In particular, we explore isometries,  the type of self-map that generate bounded composition operator, and characterize the boundedness of composition operator in terms of Poisson integral. Furthermore, we establish the relation between reproducing kernels and composition operators on Hilbert space of complex-valued harmonic functions.}
	\newline \\
	\textbf{Keywords/phrases}: Composition Operators; Integral means;  Mean-value Theorem; Poisson Integral;  Reproducing kernel. 
	\section{Introduction}
	\pagenumbering{arabic}
	Let $\mathbb{D}$ denote the unit disc in $\mathbb{C}.$ The  Hilbert space of complex-valued harmonic functions $\mathnormal{H}_{h}^{2}(\mathbb{D})$ with square summable coefficients with norm $||f||_{\mathnormal{H}_{h}^{2}(\mathbb{D})}$ has been introduced respectively in \cite{GTGH} as follows:
	\begin{center}
		$\mathnormal{H}_{h}^{2}(\mathbb{D})= \lbrace f:\mathbb{D} \rightarrow \mathbb{C}: f(z)= \sum_{n=0}^{\infty} a_{n}z^{n} + \overline {\sum_{n=0}^{\infty} {b_{n}z^{n}}}$ with $\sum_{n=0}^{\infty} |a_{n}|^{2}+|b_{n}|^{2} <\infty \rbrace,$
	\end{center}  
	
	\begin{center}
		$||f||_{\mathnormal{H}_{h}^{2}(\mathbb{D})}=(\sum_{n=0}^{\infty}|a_{n}|^{2}+|b_{n}|^{2})^{\frac{1}{2}} < \infty.$ 
	\end{center}
It is known that $\mathnormal{H}_{h}^{2}(\mathbb{D})$ is the Hilbert space of complex-valued harmonic functions under the above norm (see, \cite{GTGH}). If  $\phi:\mathbb{D} \rightarrow \mathbb{D}$ is analytic, then the composition operator $\mathnormal{C}_{\phi}:\mathnormal{H}_{h}^{2}(\mathbb{D}) \rightarrow \mathnormal{H}_{h}^{2}(\mathbb{D})$ is defined by  
	\begin{center}
		$\mathnormal{C}_{\phi}f=\mathnormal{C}_{\phi}(h+\overline{g})=(h+\overline{g})$ o $\phi,$ for  $f=h+\overline{g} \in \mathnormal{H}_{h}^{2}(\mathbb{D}).$
	\end{center} 
 Boundedness, and other properties of composition operators on Hilbert space of analytic function in the unit disc have been characterized in many contexts (see, for example, \cite{CTW}, \cite{CCMB19}, \cite{GPTM},\cite{MMVD}, \cite{MBR} and \cite{NEA}). Broadly, one is interested in extracting properties of $\mathnormal{C}_{\phi}$ acting on space of analytic function in the unit disc from function theoretic properties of $\phi$ and vice-versa. For general information about composition operators on Hilbert space (weighted Hilbert space) of analytic function in $\mathbb{D}$, we refer the reader to the monographs by (\cite{AIR}, \cite{CCMB95}, \cite{HM}, \cite{LK}, \cite{SJH93} and \cite{ZK}). Moreover, many interesting results have been found for composition operators of analytic function on Hardy, Hardy-Hilbert and Bergman spaces (refer, \cite{BDM}, \cite{MARA}, \cite{RHB} and \cite{SJH98}). \newline 
 
In this paper, we extend the study of bounded composition operators on Hilbert space of analytic function to bounded composition operators on Hilbert space of complex-valued harmonic functions in the unit disc. Analogous to the analytic case we use properties of $\phi$ to determine first when $\mathnormal{C}_{\phi}$ is an isometry, and then when $\mathnormal{C}_{\phi}$ is bounded. Moreover, we characterize the boundedness of $\mathnormal{C}_{\phi}$ in terms of Poisson integral and also study the relationship between reproducing kernel and composition operator on $\mathnormal{H}_{h}^{2}(\mathbb{D}).$ Therefore, in this paper we explore, characterize and study relation between operators with the following leading problem statements:
	
	\begin{itemize}
		\item [1.] Explore the type of self-map $\phi$ on $\mathbb{D}$ so that the composition operator $\mathnormal{C}_{\phi}$ is bounded on Hilbert space of complex-valued harmonic functions in the unit disc and vice-versa.  
		\item [2.] Characterize the boundedness of the composition operator $\mathnormal{C}_{\phi}$ in terms of Poisson integral on Hilbert space of complex-valued harmonic functions in the unit disc.
		\item [3.] Study the relationship between reproducing kernels and composition operators on Hilbert space of complex-valued harmonic functions in the unit disc.
	\end{itemize}
	\section{Preliminaries}	
In this section we review some important concepts that we will use to prove main results in section 3. We begin by stating the well known results, some useful definitions, and theorems. \newline

The Hardy space $\mathnormal{H^2}$ is the set of analytic functions on the open unit disc, denoted by $\mathbb{D}:=\lbrace z \in \mathbb{C} : |z|<1 \rbrace$ , whose power series representation has square summable coefficients. That is, 
\begin{center} 
	$\mathnormal{H^2} (\mathbb{D}) = \{f(z)= \sum_{n=0}^{\infty} a_n z^n: \sum_{n=0}^{\infty} \vert a_n \rvert ^2 < \infty\}$.      
\end{center}
A composition operator generated by an analytic self-map $\phi : \mathbb{D} \rightarrow \mathbb{D}$ is defined by
\begin{center}
	$\mathnormal{C}_{\phi}f=fo\phi$, $f \in \mathnormal{H}^{2}(\mathbb{D}).$
\end{center}
A central question of interest is identifying the additional conditions that must be imposed on a function $\phi$ for the associated composition operator $\mathnormal{C}_{\phi}$ to act boundedly on the Hardy space $\mathnormal{H}^{2}(\mathbb{D}).$ A foundational result in this context is Littlewood's subordination Theorem, which establishes that every analytic self-map $\phi$ of the unit disc induces a bounded composition operator $\mathnormal{C}_{\phi}:\mathnormal{H}^{2}(\mathbb{D}) \rightarrow \mathnormal{H}^{2}(\mathbb{D}).$ In Shapiro \cite{SJH20}, the norm of a composition operator was studied in relation to inner functions. It was shown that the norms of composition opeartors on $\mathnormal{H}^{2}(\mathbb{D})$ are bounded both above and below by quantities that depends solely on the value of $\phi$ at zero, i.e., $\phi.$ A natural question then arises: under what conditions do these operators attain their upper norm bounds? More precisely, which analytic self-maps $\phi$ maximize the operator norm of $\mathnormal{C}_{\phi}?$ Shapiro \cite{SJH20} addressed this question and demonstrated that the norm of $\mathnormal{C}_{\phi}$ is maximal if and only if $\phi$ is an inner function.    
\newline \\ 
In general, for the Hardy space $\mathnormal{H^2}(\mathbb{D})$, the Littlewood subordination theorem, along with some straightforward computations involving variable changes induced by automorphisms of the unit disc, implies that the composition operator $\mathnormal{C}_{\phi}$ is bounded for every analytic function $\phi$ that maps the unit disc $\mathbb{D}$ into itself. This approach also leads to the following estimate for the norm of composition operators on $\mathnormal{H^2}(\mathbb{D}):$   
\begin{center}
	$(\frac{1}{1-|\phi(0)|^{2}})^{\frac{1}{2}} \le ||\mathnormal{C}_{\phi}|| \le (\frac{1+|\phi(0)|}{1-|\phi(0)|})^{\frac{1}{2}}$.
\end{center}  
This is the result highlights the connection between the behavior of the operator $\mathnormal{C}_{\phi}$ and the analytic and geometric properties of the symbol $\phi.$ When operator theorists examine a new operator, they typically begin by exploring questions of boundedness and compactness, as well as investigating its norm, spectrum, and adjoint. Although a complete understanding is still developing, significant progress has been made, and it is anticipated that the answers will be expressed in terms of the analytic and geometric characterstics of $\phi.$     
\newline
\begin{definition}
	{\cite{GTGH}.} Let $f(z)=h(z)+\overline{g(z)}$, where $h(z)=\sum_{n=o}^{\infty} a_nz^n$ and $g(z)={\sum_{n=o}^{\infty}b_nz^n}$ are analytic. Then the Hilbert space of complex-valued harmonic functions in the unit disc denoted by $\mathnormal{H}_{h}^{2}(\mathbb{D})$ is defined as
	\begin{center}
		$\mathnormal{H}_{h}^{2}(\mathbb{D})= \lbrace f:\mathbb{D} \rightarrow \mathbb{C}: f(z)= \sum_{n=0}^{\infty} a_{n}z^{n} + \overline {\sum_{n=0}^{\infty} {b_{n}z^{n}}}$ with $\sum_{n=0}^{\infty} |a_{n}|^{2}+|b_{n}|^{2} <\infty \rbrace$.
	\end{center}  
\end{definition}
	The norm defined on Hilbert space of complex-valued harmonic functions in $\mathbb{D}$ has another equivalent representation in terms of integral means.
	\begin{thm}
		{\cite{GTGH}.}The norm defined on $\mathnormal{H}_{h}^{2}(\mathbb{D})$ has represented by  
		\begin{center}
			$\mathnormal{M}_{2}^{2} (f,r) = \frac{1}{2\pi} \int_{-\pi}^{\pi} |f(re^{i\theta})|^{2} d\theta$
		\end{center}
		where $\mathnormal{M}_{2}^{2} (f,r)$ denote the integral mean, $f$ is assumed to be a complex-valued harmonic functions on $\mathbb{D}$ and $0 \leq r < 1$. 		
	\end{thm}
	\begin{thm}
		{\cite{GTGH}.}(Reproducing Kernels). For $\alpha \in \mathbb{D}$, the function $\mathnormal{K}_{\alpha}$ defined by 
		\begin{center}
			$\mathnormal{K}_{\alpha}(z)=\frac{1}{1-\bar{\alpha}z}+\frac{1}{1-\alpha\bar{z}}$
		\end{center}
		has the property that $\langle f, \mathnormal{K}_{\alpha} \rangle=f(\alpha)$ for every $f$ in $\mathnormal{H}_{h}^{2}(\mathbb{D})$.  
	\end{thm}
	The Mean-value theorem, Poisson kernel and Poisson's integral formula have been defined in various sources (refer, for example, \cite{ADR}, \cite{CJB} and \cite{MARA}) for the details.  
	\begin{thm}
	\cite{CJB} (Mean-value Theorem). Let $u:\Omega \rightarrow \mathbb{R}$ be a harmonic function and let $\overline{B(a;r)}$ be  a closed disc contained in $\Omega$. If $\gamma$ is the circle $|z-a|=r$, then
		\begin{center}
			$u(a)=\frac{1}{2\pi} \int_{0}^{2\pi} u(a+re^{i\theta})d\theta$.
		\end{center}
	\end{thm}
	Note that in \cite{ADR}, If $z=re^{i\theta}$, then
	\begin{center}
		$P_{z}(t)=\frac{1-r^{2}}{|e^{it}-re^{i\theta}|^{2}}=\frac{1-r^{2}}{|e^{i(t-\theta}-r|^{2}}=P_{r}(t-\theta)$
	\end{center}
	and since $|e^{i(t-\theta}-r|^{2}=1-2rcos(t-\theta)+r^{2},$ we see that 
	\begin{center}
		$P_{r}(t-\theta)=\frac{1-r^{2}}{1-2rcos(t-\theta)+r^{2}}=\frac{1-r^{2}}{1-2rcos(\theta-t)+r^{2}}=P_{r}(\theta-t).$
	\end{center} 
	Thus for $0 \leq r<1,$ $P_{r}(x)$ is an even function of $x.$ Note also that 
	$P_{r}(x)$ is positive and decreasing on $[0,2\pi].$ In complex plane, the Poisson kernel for the unit disc is given by 
	\begin{center}
		$P_{r}(\theta)=\sum_{n=-\infty}^{\infty} r^{|n|}e^{in\theta}=\frac{1-r^{2}}{1-2rcos\theta+r^{2}}=\Re(\frac{1+re^{i\theta}}{1+re^{i\theta}}),$ $0\leq r<1.$
	\end{center} 
	\begin{thm} \cite{ADR}
		The Poisson kernel $P_{z}(t)=\frac{1-|z|^{2}}{|e^{it}-z|^{2}}$ is a harmonic function of $z.$ Moreover, 
		\begin{center} 
			$P_{z}(t) \geq 0$ and $\frac{1}{2\pi}\int_{0}^{2\pi}P_{r}(t-\theta) d\theta=1.$
		\end{center}	
	\end{thm}
	\begin{thm} \cite{CJB}
Let $\mathbb{D}=\lbrace z:|z|<1 \rbrace$ and suppose that $f:\partial \mathbb{D} \rightarrow \mathbb{R}$ is continuous function. Then there is a continuous function $\mathnormal{u}:\overline{\mathbb{D}} \rightarrow \mathbb{R}$ such that $u(z)=f(z)$ for $z$ in $\partial \mathbb{D},$ $u$ is harmonic in $\mathbb{D},$ and for $0 \leq r<1,$ $0 \leq \theta \leq 2\pi, u$ is uniquely given by 
		
			\begin{center}
		$u(re^{i\theta})=\frac{1}{2\pi}\int_{-\pi}^{\pi}P_{r}(\theta-t) f(e^{it}) dt.$
			\end{center} 
	\end{thm}

	\begin{cor}
\cite{CJB} If $u:\overline{\mathbb{D}} \rightarrow \mathbb{R}$ is a continuous function that is harmonic in $\mathbb{D}$, then 
 \begin{center}
$u(re^{i\theta})=\frac{1}{2\pi} \int_{-\pi}^{\pi} P_{r}(\theta-t) u(e^{it}) dt$  \end{center} 
for $0 \leq r<1$ and all $\theta.$ Moreover, $u$ is the real part of the analytic function 
		\begin{center}
			$f(z)=\frac{1}{2\pi} \int_{-\pi}^{\pi} \frac{e^{it}+z}{e^{it}-z} u(e^{it})dt.$
		\end{center}
	\end{cor} 
	\section{Main Results}
	In this section, we explore  self-maps of analytic function $\phi$ in $\mathbb{D}$ such that the composition operator $\mathnormal{C}_{\phi}$ is bounded on Hilbert space of complex-valued harmonic functions and vice-versa. Moreover, we prove the boundedness of composition operator $\mathnormal{C}_{\phi}$ in terms of Poisson integral and establish the relation between reproducing kernels and composition operators on $\mathnormal{H}_{h}^{2}(\mathbb{D}).$  
	\subsection{Boundedness of $\mathnormal{C}_{\phi}$ on $\mathnormal{H}_{h}^{2}(\mathbb{D})$} 
	Suppose $f=h+\overline{g} \in \mathnormal{H}_{h}^{2}(\mathbb{D})$ is of the form 
	\begin{center}
		$f(z)=\sum_{n=0}^{\infty} a_nz^{n}+\overline{\sum_{n=0}^{\infty} b_nz^{n}}, z \in \mathbb{D}.$ 
	\end{center}
	 Then for fixed $k \in \mathbb{N},$
	\begin{center}
		$f(\phi(z))=\sum_{n=0}^{\infty} a_nz^{nk}+\overline{\sum_{n=0}^{\infty} b_nz^{nk}}, z \in \mathbb{D},$ 	
	\end{center}
	is another function in $\mathnormal{H}_{h}^{2}(\mathbb{D}),$ is of the same norm as $f.$ In this case observe that,  $\phi(z)=z^{k},$ and the associated composition operator $\mathnormal{C}_{\phi}$ is given by, 
	\begin{center}
		$\mathnormal{C}_{\phi}f=\mathnormal{C}_{\phi}(h+\overline{g})=(h+\overline{g})o\phi(z), z \in \mathbb{D},$ and
		$||\mathnormal{C}_{\phi}f||_{\mathnormal{H}_{h}^{2}(\mathbb{D})}^{2}=\sum_{n=0}^{\infty}|a_{n}|^{2}+|b_{n}|^{2}=||f||_{\mathnormal{H}_{h}^{2}(\mathbb{D})}^{2},$
	\end{center}
	so that $\mathnormal{C}_{\phi}$ is an isometry on $\mathnormal{H}_{h}^{2}(\mathbb{D}).$  \\
	Now one can ask what other isometries on $\mathnormal{H}_{h}^{2}(\mathbb{D})$ are there. This can be answered in the following Theorem:  
	\begin{thm}
		An analytic function $\phi: \mathbb{D} \rightarrow \mathbb{D}$ generates a bounded composition operator $\mathnormal{C}_{\phi} : \mathnormal{H}_{h}^{2}(\mathbb{D}) \rightarrow \mathnormal{H}_{h}^{2}(\mathbb{D})$ which is an isometry provided that
		\begin{itemize}
			\item [(i).] $\phi(z)= e^{i\theta} z, \theta \in \mathbb{R},$ 
			\item [(ii).] $\phi(z)=\alpha z^{k}$ for $k \geq 1$ if and only if   $|\alpha|=1,$ and 
			\item [(iii).] $\phi(z)=\frac{a-z}{1-\bar{a}z}, \hspace{1mm} a \in \mathbb{D}$.
		\end{itemize}
	\end{thm}
	\begin{proof} Throughout the proof of this theorem we consider $f=h+\overline{g} \in \mathnormal{H}_{h}^{2}(\mathbb{D}).$
		\begin{itemize}
			\item [(i).] Observe that $\phi(z)= e^{i\theta} z, \theta \in \mathbb{R}$ is a self-map of analytic function in $\mathbb{D}$ because for $z \in \mathbb{D},$ $|\phi(z)|=|e^{i\theta}z|=|z|<1.$ It follows that, 
			\begin{center}
				$\mathnormal{C}_{\phi}f=(h+\overline{g})o\phi=\sum_{n=0}^{\infty}a_n e^{in\theta}z^{n}+\overline{\sum_{n=0}^{\infty} b_n e^{in\theta}z^{n}},$ and computing the norm we get,
			\end{center}
			
			\begin{center}
				$||\mathnormal{C}_{\phi}f||_{\mathnormal{H}_{h}^{2}(\mathbb{D})}^{2}=\sum_{n=0}^{\infty} |a_n e^{in\theta}|^{2}+|b_n e^{in\theta}|^{2}=||f||_{\mathnormal{H}_{h}^{2}(\mathbb{D})}^{2}.$
			\end{center} 
			Thus, $\mathnormal{C}_{\phi}$ is an isometry on $\mathnormal{H}_{h}^{2}(\mathbb{D}).$ 
			\item [(ii).] Clearly, $\phi(z)=\alpha z^{k}$ is a self-map of $\mathbb{D}$ for $k \geq 1$. Then 
			\begin{center}
				$\mathnormal{C}_{\phi}f= (h+\overline{g})(\alpha z^{k}) =\sum_{n=0}^{\infty} a_{n}\alpha^{n} z^{kn} + \overline {\sum_{n=0}^{\infty} {b_{n}\alpha^n z^{kn}}}.$
			\end{center}
Thus computing the norm, we get			
			\begin{center}
				$||\mathnormal{C}_{\phi} f||_{\mathnormal{H}_{h}^{2}(\mathbb{D})}^{2}=\sum_{n=0}^{\infty} |a_{n}\alpha^{n}|^2 + |b_{n}\alpha^n|^2=\sum_{n=0}^{\infty} |\alpha|^{2n}(|a_n|^2 + |b_{n}|^2)$.	
			\end{center}
			Hence, $||\mathnormal{C}_{\phi} f||_{\mathnormal{H}_{h}^{2}(\mathbb{D})}^{2} = ||f||_{\mathnormal{H}_{h}^{2}(\mathbb{D})}^{2}$ if and only if $|\alpha| = 1.$ Which shows that, $\mathnormal{C}_{\phi}$ is an isometry on $\mathnormal{H}_{h}^{2}(\mathbb{D}).$
			\item [(iii).] For $\phi(z)=\frac{a-z}{1-\bar{a}z}, a \in \mathbb{D},$ observe that $|\phi(z)|=\frac{|a-z|}{|1-\bar{a}z|} < 1$ since $|a|<1$ and $|z|<1$, $|a-z|$ is finite and $|1-\bar{a}z|>0.$ Moreover, the M$\ddot{o}$bius transformation preserves the unit disc $\mathbb{D}$ onto itself. Thus, $\phi(z)$ is a self-map of $\mathbb{D}$. Again, we need to show that $\mathnormal{C}_{\phi}f \in \mathnormal{H}_{h}^{2}(\mathbb{D})$ for all $f \in \mathnormal{H}_{h}^{2}(\mathbb{D})$. By theorem (2.2), the norm of $f$ on $\mathnormal{H}_{h}^{2}(\mathbb{D})$ in terms of mean integral formula is given by 
			\begin{center}
				$||f||_{\mathnormal{H}_{h}^{2}(\mathbb{D})}^{2} =\mathnormal{M}_{2}^{2}(f,r)=\frac{1}{2\pi} \int_{-\pi}^{\pi}|f(re^{i\theta})|^2d \theta$. 
			\end{center}
			This implies
			\begin{center}
				$||\mathnormal{C}_{\phi}f||_{\mathnormal{H}_{h}^{2}(\mathbb{D})}^{2} =\mathnormal{M}_{2}^{2}(fo\phi,r)=\frac{1}{2\pi} \int_{-\pi}^{\pi}|f(\phi(re^{i\theta}))|^2d \theta$. 
			\end{center} 
			Again, since $f$ is harmonic and $\phi(z)$ is analytic in $\mathbb{D}$ their composition $f(\phi(z))$ is also harmonic in $\mathbb{D}$.
			By the substitution formula for Hilbert space of complex-valued harmonic functions norms and the fact that Mobius transformations preserves measures on the unit disc, it is known that 
			\begin{center}
				$||fo\phi||_{\mathnormal{H}_{h}^{2}(\mathbb{D})}^{2} =\mathnormal{M}_{2}^{2}(f,r)=\frac{1}{2\pi} \int_{-\pi}^{\pi}|f(re^{i\theta})|^2d \theta$.
			\end{center}
	This follows because the integral of $|f(\phi(re^{i\theta}))|^2$ with respect to $\theta$ remain unchanged due to the invariance property of M$\ddot{o}$bius transformations. So,  
			\begin{center}
				$\frac{1}{2\pi} \int_{-\pi}^{\pi}|f(\phi(re^{i\theta}))|^2 d \theta =\frac{1}{2\pi} \int_{-\pi}^{\pi}|f(re^{i\theta})|^2d \theta$.	
			\end{center}
Thus, $\mathnormal{C}_{\phi}f \in \mathnormal{H}_{h}^{2}(\mathbb{D})$ and 
			\begin{center}
$||\mathnormal{C}_{\phi}f||_{\mathnormal{H}_{h}^{2}(\mathbb{D})} = ||f||_{\mathnormal{H}_{h}^{2}(\mathbb{D})}.$ 
			\end{center}
This implies that $\mathnormal{C}_{\phi}$ is bounded with $||\mathnormal{C}_{\phi}||_{\mathnormal{H}_{h}^{2}(\mathbb{D})}=1,$ so that it is an isometry.
		\end{itemize}
	\end{proof}
	\begin{thm}
		Suppose $\phi$ is an analytic function on $\mathbb{D}$ and $\mathnormal{C}_{\phi}$ is a bounded composition operator on $\mathnormal{H}_{h}^{2}(\mathbb{D}).$ Then $\phi$ is a self-map on $\mathbb{D}$.
	\end{thm}
	\begin{proof}
Assume $\mathnormal{C}_{\phi}:\mathnormal{H}_{h}^{2}(\mathbb{D}) \rightarrow \mathnormal{H}_{h}^{2}(\mathbb{D})$ is bounded. Then 		
$||\mathnormal{C}_{\phi}f||_{\mathnormal{H}_{h}^{2}(\mathbb{D})} \leq k  ||f||_{\mathnormal{H}_{h}^{2}(\mathbb{D})}$ for some $ k>0.$ Suppose  $\phi$ is not a self-map on $\mathbb{D}.$ Then there is $z_{0} \in \mathbb{D}$ such that $|\phi(z_{0})| \geq 1.$ By theorem (2.3), the reproducing kernel function is given by 
\begin{center}
	$K_{\alpha}(z)=\frac{1}{1-\overline{\alpha}z}+\frac{1}{1-\alpha \overline{z}},$ for $\alpha \in \mathbb{D}.$
\end{center}  
The family $\lbrace K_{\alpha} \rbrace_{\alpha \in \mathbb{D}}$ forms a dense subset in $\mathnormal{H}_{h}^{2}(\mathbb{D}),$ and their norms satisfy 
\begin{center}
	$||K_{\alpha}||_{\mathnormal{H}_{h}^{2}(\mathbb{D})}=\frac{\sqrt{2}}{\sqrt{1-|\alpha|^{2}}}.$
\end{center}
Now applying $\mathnormal{C}_{\phi}$ to $\mathnormal{K}_{\alpha},$ we obtain 
\begin{center}
	$\mathnormal{C}_{\phi}\mathnormal{K}_{\alpha}(z)=\mathnormal{K}_{\alpha}(\phi(z))=\frac{1}{1-\overline{\alpha} \phi(z)}+\frac{1}{1-\alpha \overline{\phi(z)}},$
\end{center}
and 
\begin{center}
	$||\mathnormal{C}_{\phi}\mathnormal{K}_{\alpha}(z)||_{\mathnormal{H}_{h}^{2}(\mathbb{D})}=||\mathnormal{K}_{\alpha}(\phi(z))||_{\mathnormal{H}_{h}^{2}(\mathbb{D})}=(\frac{2(1-\Re(\alpha \overline{\phi(\alpha)}))}{|1-\overline{\alpha} \phi(\alpha)|^{2}})^{\frac{1}{2}}.$
\end{center} 
If $\phi(z_{0})$ satisfies $|\phi(z_{0})| \geq 1,$ then for $\alpha$ close to $\frac{1}{\overline{\phi(z_{0})}},$ the denominator $1-\overline{\alpha} \phi(\alpha)$ approaches zero for $\alpha \in \mathbb{D},$ causing the norm unbounded. This contradicts the boundedness of $\mathnormal{C}_{\phi},$ proving that  $\phi(\mathbb{D}) \subset (\mathbb{D}).$
	\end{proof}
	Although our main focus is to study the boundedness of composition operator $\mathnormal{C}_{\phi}$ on $\mathnormal{H}_{h}^{2}(\mathbb{D}),$ but we are also interested to give an example of the type of self-map $\phi$ in which $\mathnormal{C}_{\phi}$ is non-isometry.
	\begin{Example}
		Suppose $\phi(z)= az+b$, $|a| \leq 1$ and $|b| \leq 1-|a|.$ If $f=h+\overline{g}$ is in $\mathnormal{H}_{h}^{2}(\mathbb{D}),$ then $\mathnormal{C}_{\phi}f$ is not in $\mathnormal{H}_{h}^{2}(\mathbb{D}),$ and in fact, $||\mathnormal{C}_{\phi}f||_{\mathnormal{H}_{h}^{2}(\mathbb{D})} \neq ||f||_{\mathnormal{H}_{h}^{2}(\mathbb{D})}.$ Therefore, $\mathnormal{C}_{\phi}$ is not an isometry mapping $\mathnormal{H}_{h}^{2}(\mathbb{D})$ into itself. Clearly, $\phi$ is a self-map of analytic function in $\mathbb{D},$ and one can easily show by using Binomial expansion that $\mathnormal{C}_{\phi}$ is not an isometry on  $\mathnormal{H}_{h}^{2}(\mathbb{D})$ as the norm of $\mathnormal{C}_{\phi}f$ is not the same as norm of $f.$
			\end{Example}		  
	
	\subsection{Boundedness of $\mathnormal{C}_{\phi}$ in terms of Poisson Integral on $ {\mathnormal{H}_{h}^{2}(\mathbb{D})}$}
	Now, we characterize the boundedness of composition operator $\mathnormal{C}_{\phi}$ in terms of Poisson integral, and we also establish the relationship between reproducing kernels and composition operators on
	Hilbert space of complex-valued harmonic functions in the unit disc.
	\begin{definition}
		Suppose $f \in \mathnormal{H}_{h}^{2}(\mathbb{D})$. Then the poisson integral formula for $re^{it} \in \mathbb{D}$, is defined by
		\begin{center}
			$f(re^{it})=h(re^{it})+\overline{g(re^{it})} = \frac{1}{2\pi} \int_{0}^{2\pi} P_{r}(\theta -t)h(e^{i\theta}) d\theta+\frac{1}{2\pi} \int_{0}^{2\pi} P_{r}(\theta -t)\overline{g(e^{i\theta})} d\theta$.
		\end{center} 	
	\end{definition}	
	\begin{lem}
		Suppose $f \in \mathnormal{H}_{h}^{2}(\mathbb{D}).$ Then, for $re^{it} \in \mathbb{D}$, we have 
		\begin{center}
			$|f(re^{it})|^{2} \leq \frac{1}{2\pi} \int_{0}^{2\pi} P_{r}(\theta -t)|h(e^{i\theta})|^{2} d\theta + \frac{1}{2\pi} \int_{0}^{2\pi} P_{r}(\theta -t)|g(e^{i\theta})|^{2} d\theta $
		\end{center} 
		$~~~~~~~~~~~~~~~~~~~~~~~+ 2 (\frac{1}{2\pi} \int_{0}^{2\pi} P_{r}(\theta -t)|h(e^{i\theta})|^{2}d\theta)^{\frac{1}{2}}( \frac{1}{2\pi} \int_{0}^{2\pi} P_{r}(\theta -t)|g(e^{i\theta})|^{2}d\theta)^{\frac{1}{2}}.$
	\end{lem}
	\begin{proof}
		By definition 3.4., we have
		\begin{center}
			$f(re^{it}) = \frac{1}{2\pi} \int_{0}^{2\pi} P_{r}(\theta -t)h(e^{i\theta}) d\theta+\frac{1}{2\pi} \int_{0}^{2\pi} P_{r}(\theta-t)\overline{g(e^{i\theta})} d\theta$.
		\end{center} 
		If we define the Lebesgue measure $d\mu$ by $d\mu(\theta)=\frac{1}{2\pi}P_r(\theta - t) d \theta $, then the above formula becomes,
		\begin{center}
			$f(re^{it}) = \int_{0}^{2\pi} h(e^{i\theta}) d\mu(\theta)+ \int_{0}^{2\pi} \overline{g(e^{i\theta})} d\mu(\theta)$.
		\end{center}   	
		Notice that $h(e^{i\theta})+\overline{g(e^{i\theta})} \in \mathnormal{L}^{2}(\mathnormal{S}^{1},d\mu)$, since for fixed $r<1$, $P_r(\theta-t)$ is bounded above. Applying the triangle inequality and Cauchy-schwarz inequality to the product of the functions $h(e^{i\theta})\in \mathnormal{L}^{2}(\mathnormal{S}^{1},d\mu)$, $\overline{g(e^{i\theta})} \in \mathnormal{L}^{2}(\mathnormal{S}^{1},d\mu)$, and $1 \in \mathnormal{L}^{2}(\mathnormal{S}^{1},d\mu)$, we obtain 
	\newline 
$|f(re^{it})|=|\int_{0}^{2\pi} h(e^{i\theta}) d\mu(\theta) + \int_{0}^{2\pi} \overline{g(e^{i\theta})} d\mu(\theta)| \leq |\int_{0}^{2\pi} h(e^{i\theta}) d\mu(\theta)| + |\int_{0}^{2\pi} \overline{g(e^{i\theta})} d\mu(\theta) |$	
\newline
$~~~~~~~~~~~\leq (\int_{0}^{2\pi} |h(e^{i\theta})|^{2} d\mu(\theta))^{\frac{1}{2}}(\int_{0}^{2\pi}|1|^2 d\mu(\theta))^{\frac{1}{2}} + (\int_{0}^{2\pi} |\overline{g(e^{i\theta})}|^{2} d\mu(\theta))^{\frac{1}{2}}(\int_{0}^{2\pi}|1|^2 d\mu(\theta))^{\frac{1}{2}}$	
\newline
$~~~~~~~~~~~=(\int_{0}^{2\pi} |h(e^{i\theta})|^{2} d\mu(\theta))^{\frac{1}{2}} + (\int_{0}^{2\pi} |g(e^{i\theta})|^{2} d\mu(\theta))^{\frac{1}{2}}$	
		\newline	
since $\frac{1}{2\pi} \int_{0}^{2\pi} P_r(\theta -t) d\theta=1.$ Squaring both sides of the above inequality yields 
		\newline \
	$|f(re^{it})|^{2} \leq [(\int_{0}^{2\pi} |h(e^{i\theta})|^{2} d\mu(\theta))^{\frac{1}{2}} + (\int_{0}^{2\pi} |g(e^{i\theta})|^{2} d\mu(\theta))^{\frac{1}{2}}]^{2}$ \newline
	$~~~~~~~~~~~~~=\int_{0}^{2\pi} |h(e^{i\theta})|^{2} d\mu(\theta) + \int_{0}^{2\pi} |g(e^{i\theta})|^{2} d\mu(\theta)+2(\int_{0}^{2\pi} |h(e^{i\theta})|^{2} d\mu(\theta))^{\frac{1}{2}}(\int_{0}^{2\pi} |g(e^{i\theta})|^{2} d\mu(\theta))^{\frac{1}{2}}$ \newline
$~~~~~~~~~~~~~=\frac{1}{2\pi} \int_{0}^{2\pi} P_r(\theta-t)|h(e^{i\theta})|^{2} d\theta + \int_{0}^{2\pi} P_r(\theta-t)|g(e^{i\theta})|^{2} d\theta \newline ~~~~~~~~~~~~~~+2(\frac{1}{2\pi}\int_{0}^{2\pi} P_r(\theta-t)|h(e^{i\theta})|^{2} d\theta)^{\frac{1}{2}}(\frac{1}{2\pi}\int_{0}^{2\pi}P_r(\theta-t) |g(e^{i\theta})|^{2} d\theta)^{\frac{1}{2}}.$
	\end{proof}
	We can now prove that $\mathnormal{C}_{\phi}$ is well-defined and bounded operator on $\mathnormal{H}_{h}^{2}(\mathbb{D})$. 
	\begin{thm}
		Let $\phi:\mathbb{D} \rightarrow \mathbb{D}$ be analytic. Then the composition operator $\mathnormal{C}_{\phi}$ is well-defined and bounded on $\mathnormal{H}_{h}^{2}(\mathbb{D})$. Moreover, 
		\begin{center}
			$||\mathnormal{C}_{\phi}||_{\mathnormal{H}_{h}^{2}(\mathbb{D})} \leq 2\sqrt{\frac{1+|\phi(0)|}{1-|\phi(0)|}}$.
		\end{center}
	\end{thm}
	\begin{proof}
		The expression 
		\begin{center}
			$P_{r}(\theta-t)=\frac{1-r^{2}}{1+2r cos(\theta -t)+r^{2}} \geq 0$
		\end{center} is the Poisson kernel in the unit disc. We consider real-valued functions $u$ and $v$ on $\mathbb{D}$ defined respectively by
		\begin{center}
			$u(re^{it})=\frac{1}{2\pi} \int_{0}^{2\pi} P_{r}(\theta -t)|h(e^{i\theta})|^{2} d\theta$ and $v(re^{it})=\frac{1}{2\pi} \int_{0}^{2\pi} P_{r}(\theta -t)|g(e^{i\theta})|^{2} d\theta.$
		\end{center} 
		These are harmonic functions on $\mathbb{D},$ and, by lemma (3.5), 
		\begin{center}
			$|f(re^{it})|^2 \leq u(re^{it})+v(re^{it})+2\sqrt{u(re^{it})}\sqrt{v(re^{it})},$
		\end{center} for all $re^{it} \in \mathbb{D}.$
		Since the range of $\phi$ is in $\mathbb{D}$, by the inequality above, 
		\begin{center}
			$|f(\phi(re^{it}))|^2 \leq u(\phi(re^{it}))+v(\phi(re^{it}))+2\sqrt{u(\phi(re^{it}))}\sqrt{v(\phi(re^{it}))}.$
		\end{center}
Now, multiplying both sides of this inequality by $\frac{1}{2\pi}$, integrating from $0$ to $2\pi$ and applying the Cauchy-Schwartz inequality, we obtain 
\newline 	
	
$\frac{1}{2\pi} \int_{0}^{2\pi}|f(\phi(re^{it}))|^2dt \newline ~~~~~~~~~~~ \leq \frac{1}{2\pi} \int_{0}^{2\pi}u(\phi(re^{it}))dt+\frac{1}{2\pi} \int_{0}^{2\pi}v(\phi(re^{it}))dt+2(\frac{1}{2\pi} \int_{0}^{2\pi}\sqrt{u(\phi(re^{it}))}\sqrt{v(\phi(re^{it}))}dt)$
\newline 
$~~~~~~~~~~~~~\leq \frac{1}{2\pi} \int_{0}^{2\pi}u(\phi(re^{it}))dt+\frac{1}{2\pi} \int_{0}^{2\pi}v(\phi(re^{it}))dt+2\sqrt{\frac{1}{2\pi} \int_{0}^{2\pi}u(\phi(re^{it}))dt}\sqrt{\frac{1}{2\pi} \int_{0}^{2\pi}v(\phi(re^{it}))dt}$ \newline 
	
Since $u$ and $v$ are harmonic and $\phi$ is analytic, it follows that $u$ o $ \phi$ and $v$ o $\phi$ are also harmonic. Hence, by the mean value property of harmonic functions
		\begin{center}
$u(\phi(0))=\frac{1}{2\pi} \int_{0}^{2\pi}u(\phi(re^{it}))dt$ and  $v(\phi(0))=\frac{1}{2\pi} \int_{0}^{2\pi}v(\phi(re^{it}))dt.$
		\end{center}
Therefore 
		\begin{center}
			$\frac{1}{2\pi} \int_{0}^{2\pi}|f(\phi(re^{it}))|^2dt \leq u(\phi(0))+v(\phi(0))+2\sqrt{u(\phi(0))}\sqrt{v(\phi(0))}$	
		\end{center}
		Thus, $f$ o $\phi$ is in $\mathnormal{H}_{h}^{2}(\mathbb{D})$ and $\mathnormal{C}_{\phi}$ is well-defined. 
		To check that $\mathnormal{C}_{\phi}$ is bounded, observe that the last inequality implies that 
		\begin{center}
			$||\mathnormal{C}_{\phi}f||_{\mathnormal{H}_{h}^{2}(\mathbb{D})}^{2} \leq u(\phi(0))+v(\phi(0))+2\sqrt{u(\phi(0))}\sqrt{v(\phi(0))}$.
		\end{center}
		Notice also that 
		\begin{center}
			$P_{r}(\theta -t)=\frac{1-r^2}{1-2rcos(\theta -t)+r^2} \leq \frac{1-r^2}{(1-r)^2}=\frac{1+r}{1-r}$.
		\end{center}
		Since 
		\begin{center}
			${u(re^{it})}=\frac{1}{2\pi} \int_{0}^{2\pi} P_{r}(\theta -t)|h(e^{i\theta})|^2 d\theta$ and ${v(re^{it})}=\frac{1}{2\pi} \int_{0}^{2\pi} P_{r}(\theta -t)|g(e^{i\theta})|^{2} d\theta$,  		
		\end{center}
		it follows that 
		\begin{center}
			${u(re^{it})} \leq (\frac{1+r}{1-r})\frac{1}{2\pi} \int_{0}^{2\pi} |h(e^{i\theta})|^2 d\theta =(\frac{1+r}{1-r}) ||h||_{\mathnormal{H}_{h}^{2}(\mathbb{D})}^{2} $, 
		\end{center} 
		and 
		\begin{center}
			${v(re^{it})} \leq (\frac{1+r}{1-r})\frac{1}{2\pi} \int_{0}^{2\pi} |g(e^{i\theta})|^2 d\theta =(\frac{1+r}{1-r})||g||_{\mathnormal{H}_{h}^{2}(\mathbb{D})}^{2} $, 
		\end{center}
		in other words, 
		\begin{center}
			$u(z) \leq (\frac{1+|z|}{1-|z|})||h||_{\mathnormal{H}_{h}^{2}(\mathbb{D})}^{2}$ and $v(z) \leq (\frac{1+|z|}{1-|z|})||g||_{\mathnormal{H}_{h}^{2}(\mathbb{D})}^{2}$
		\end{center}
		for every $z \in \mathbb{D}$. 
		In particular, 
		\begin{center}
			$u(\phi(0)) \leq (\frac{1+|\phi(0)|}{1-|\phi(0)|})||h||_{\mathnormal{H}_{h}^{2}(\mathbb{D})}^{2}$ and $v(\phi(0)) \leq (\frac{1+|\phi(0)|}{1-|\phi(0)|})||g||_{\mathnormal{H}_{h}^{2}(\mathbb{D})}^{2}.$
		\end{center}
		Hence, \begin{center}
			$||\mathnormal{C}_{\phi}f||_{\mathnormal{H}_{h}^{2}(\mathbb{D})}^{2} \leq (\frac{1+|\phi(0)|}{1-|\phi(0)|})||h||_{\mathnormal{H}_{h}^{2}(\mathbb{D})}^{2} + (\frac{1+|\phi(0)|}{1-|\phi(0)|})||g||_{\mathnormal{H}_{h}^{2}(\mathbb{D})}^{2} + 2 \sqrt{(\frac{1+|\phi(0)|}{1-|\phi(0)|})||h||_{\mathnormal{H}_{h}^{2}(\mathbb{D})}^{2}}\sqrt{(\frac{1+|\phi(0)|}{1-|\phi(0)|})||g||_{\mathnormal{H}_{h}^{2}(\mathbb{D})}^{2}} $.
		\end{center}
		\begin{center}
			$=(\frac{1+|\phi(0)|}{1-|\phi(0)|})||h||_{\mathnormal{H}_{h}^{2}(\mathbb{D})}^{2} + (\frac{1+|\phi(0)|}{1-|\phi(0)|})||g||_{\mathnormal{H}_{h}^{2}(\mathbb{D})}^{2} + 2 {(\frac{1+|\phi(0)|}{1-|\phi(0)|})}||h||_{\mathnormal{H}_{h}^{2}(\mathbb{D})}||g||_{\mathnormal{H}_{h}^{2}(\mathbb{D})}$
		\end{center}
		\begin{center}
			$=(\frac{1+|\phi(0)|}{1-|\phi(0)|})(||h||_{\mathnormal{H}_{h}^{2}(\mathbb{D})} + ||g||_{\mathnormal{H}_{h}^{2}(\mathbb{D})})^2$.
		\end{center}
		Therefore $\mathnormal{C}_{\phi}$ is bounded and 
		\begin{center}
			$||\mathnormal{C}_{\phi}f||_{\mathnormal{H}_{h}^{2}(\mathbb{D})} \leq \sqrt{\frac{1+|\phi(0)|}{1-|\phi(0)|}}(||h||_{\mathnormal{H}_{h}^{2}} + ||g||_{\mathnormal{H}_{h}^{2}}) \leq \sqrt{\frac{1+|\phi(0)|}{1-|\phi(0)|}} 2||f||_{\mathnormal{H}_{h}^{2}(\mathbb{D})}$.	
		\end{center}
		Thus, 
		\begin{center}
			$||\mathnormal{C}_{\phi}||_{\mathnormal{H}_{h}^{2}(\mathbb{D})} \leq 2\sqrt{\frac{1+|\phi(0)|}{1-|\phi(0)|}}$.	
		\end{center}
	\end{proof}
Reproducing kernels give a lot of information about composition operators and very often, calculations with kernel functions give ways to connect the analytic and geometric properties of $\phi$ with the operator properties of $\mathnormal{C}_{\phi}$.  
	\begin{lem}
		If $\mathnormal{C}_{\phi}$ is a composition operator and $\mathnormal{K}_{\alpha}$ is a reproducing kernel function, then $\mathnormal{C}_{\phi}^{*}\mathnormal{K}_{\alpha}=\mathnormal{K}_{\phi(\alpha)}.$
	\end{lem}
	\begin{thm}
		For every composition operator $\mathnormal{C}_{\phi}$, 
		\begin{center}
			$\frac{1}{\sqrt{1-|\phi(0)|^2}} \leq ||\mathnormal{C}_{\phi}||_{\mathnormal{H}_{h}^{2}(\mathbb{D})} \leq \frac{4}{\sqrt{1-|\phi(0)|^2}}$.
		\end{center}	
	\end{thm}
	\begin{proof}
		Using lemma (3.7) with $\alpha=0$ yields 
		\begin{center}
			$\mathnormal{C}_{\phi}^{*}\mathnormal{K}_{0}=\mathnormal{K}_{\phi(0)}$
		\end{center} 
	Note that in \cite{GTGH} 
		\begin{center}
			$||\mathnormal{K}_{\alpha}||_{\mathnormal{H}_{h}^{2}(\mathbb{D})}^{2}=2(\frac{1}{1-|\alpha|^{2}})$,
		\end{center}
		and therefore 
		\begin{center}
			$||\mathnormal{K}_{0}||_{\mathnormal{H}_{h}^{2}(\mathbb{D})}=\sqrt{2}$ and $||\mathnormal{K}_{\phi(0)}||_{\mathnormal{H}_{h}^{2}(\mathbb{D})}=\frac{\sqrt{2}}{\sqrt{1-|\phi(0)|^2}}$.
		\end{center} 
		Since \begin{center}
			$||\mathnormal{K}_{\phi(0)}||_{\mathnormal{H}_{h}^{2}(\mathbb{D})}=||\mathnormal{C}_{\phi}^{*}\mathnormal{K}_{0}||_{\mathnormal{H}_{h}^{2}(\mathbb{D})} \leq ||\mathnormal{C}_{\phi}^{*}||_{\mathnormal{H}_{h}^{2}(\mathbb{D})}||\mathnormal{K}_{0}||_{\mathnormal{H}_{h}^{2}(\mathbb{D})}$, 
		\end{center} 
		it follows that 
		\begin{center}
			$\frac{\sqrt{2}}{\sqrt{1-|\phi(0)|^2}} \leq {\sqrt{2}}||\mathnormal{C}_{\phi}^{*}||_{\mathnormal{H}_{h}^{2}(\mathbb{D})}={\sqrt{2}}||\mathnormal{C}_{\phi}||_{\mathnormal{H}_{h}^{2}(\mathbb{D})}$.
		\end{center}
		This implies that 
		\begin{center}
			$\frac{1}{\sqrt{1-|\phi(0)|^2}} \leq ||\mathnormal{C}_{\phi}||_{\mathnormal{H}_{h}^{2}(\mathbb{D})}.$
		\end{center}
		To prove the other inequality, we begin with the result from theorem (3.6): 
		\begin{center}
			$||\mathnormal{C}_{\phi}||_{\mathnormal{H}_{h}^{2}(\mathbb{D})} \leq 2\sqrt{\frac{1+|\phi(0)|}{1-|\phi(0)|}}$.
		\end{center}
		Observe that, for $0\leq r<1$, we have the inequality 
		\begin{center}
			$\sqrt{\frac{1+r}{1-r}}=\sqrt{\frac{(1+r)^{2}}{1-r^{2}}}=\frac{1+r}{\sqrt{1-r^{2}}} \leq \frac{2}{\sqrt{1-r^{2}}}$.	
		\end{center}
		It follows that 
		\begin{center}
			$||\mathnormal{C}_{\phi}||_{\mathnormal{H}_{h}^{2}(\mathbb{D})} \leq 2\sqrt{\frac{1+|\phi(0)|}{1-|\phi(0)|}} \leq 2(\frac{2}{\sqrt{1-|\phi(0)|^{2}}})=\frac{4}{\sqrt{1-|\phi(0)|^{2}}}.$
		\end{center}
		Thus, the result follows.
	\end{proof}
	\begin{cor}
		If $\phi(0)=0$, then $1 \leq ||\mathnormal{C}_{\phi}||_{\mathnormal{H}_{h}^{2}(\mathbb{D})} \leq 4$. 
	\end{cor}

	\begin{thm}
		An operator $\mathnormal{T}$ on $\mathnormal{H}_{h}^{2}(\mathbb{D})$ is a composition operator if and only if the adjoint operator $\mathnormal{T}^{*}$ maps the set of reproducing kernels into itself. 
	\end{thm}
	\begin{proof}	
		We showed in lemma (3.7) that $\mathnormal{T}^{*}\mathnormal{K}_{\alpha}=\mathnormal{K}_{\phi(\alpha)}$ when $\mathnormal{T}=\mathnormal{C}_{\phi}$. Conversely, suppose that for each $\alpha \in \mathbb{D}$, $\mathnormal{T}^{*}\mathnormal{K}_{\alpha}=\mathnormal{K}_{\phi(\alpha^{'})}$ for some $\alpha^{'} \in \mathbb{D}$. Define $\phi : \mathbb{D} \rightarrow \mathbb{D}$ by 
		\begin{center}
			$\phi(\alpha)=\alpha^{'}.$
		\end{center}
		Notice that, for $f \in \mathnormal{H}_{h}^{2}(\mathbb{D})$, 
		\begin{center}
			$\langle \mathnormal{T}f, \mathnormal{K}_{\alpha}\rangle = \langle f, \mathnormal{T}^{*}\mathnormal{K}_{\alpha}\rangle= \langle f, \mathnormal{K}_{\phi(\alpha)}\rangle = f(\phi(\alpha)) $.
		\end{center}
		If we take $f(z)=z+\bar{z}$, then $g=\mathnormal{T}f$ is in $\mathnormal{H}_{h}^{2}(\mathbb{D})$, and is thus analytic. But then, by the above equation we have 
		\begin{center}
			$g(\alpha)= \langle g, \mathnormal{K}_{\alpha} \rangle=\langle \mathnormal{T}f, \mathnormal{K}_{\alpha} \rangle=f(\phi(\alpha))=\phi(\alpha)$. 
		\end{center}	
		Therefore, $g=\phi$ and $\phi$ is analytic, so the composition operator $\mathnormal{C}_{\phi}$ is well-defined and bounded. It follows that $\mathnormal{T}=\mathnormal{C}_{\phi}$, since 
		\begin{center}
			$(\mathnormal{T}f)(\alpha)=\langle \mathnormal{T}f, \mathnormal{K}_{\alpha} \rangle= f(\phi(\alpha))=(\mathnormal{C}_{\phi}f)(\alpha)$ 
		\end{center} 
		for all $f$ in $\mathnormal{H}_{h}^{2}(\mathbb{D})$.	
	\end{proof}
	\subsection*{Conclusion}
	In this paper we introduced composition operators on Hilbert space of complex-valued harmonic functions in the unit disc $\mathbb{D}$. Particularly, we found analytic self map $\phi$ on unit disc for which the composition operator is an isometry and demonstrated with an example $\phi$ for which  $\mathnormal{C}_{\phi}$ is not an isometry. We also characterized the boundedness of $\mathnormal{C}_{\phi}$ in terms of analytic self-map $\phi,$  Poisson integral, and established the relationship between composition operator and reproducing kernel on Hilbert space of complex-valued harmonic functions.	
	\bibliographystyle{plain}
		
\end{document}